\documentclass{amsart}

\usepackage{amsrefs}
\usepackage{amscd}
\usepackage{amssymb}
\usepackage[cmtip,all]{xy}
\usepackage[mathscr]{eucal}
\usepackage{mathrsfs}
\usepackage{dsfont}

\newtheorem{theorem}{Theorem}[section]
\newtheorem{lemma}[theorem]{Lemma}
\newtheorem{proposition}[theorem]{Proposition}
\newtheorem{corollary}[theorem]{Corollary}

\newtheorem{definition}[theorem]{Definition}

\theoremstyle{remark}
\newtheorem{remark}[theorem]{Remark}
\newtheorem{Example}[theorem]{Example}
\numberwithin{equation}{section}

\newcommand{\cc}{\mathbf{c}}

\newcommand{\aA}{\mathbb{A}}

\newcommand{\CC}{\mathbb{C}}

\newcommand{\cO}{\mathcal{O}}

\newcommand{\RR}{\mathbb{R}}

\DeclareMathOperator{\spec}{Spec}
\DeclareMathOperator{\Coh}{Coh}
\DeclareMathOperator{\ob}{ob}
\DeclareMathOperator{\eu}{Eu}
\DeclareMathOperator{\loc}{loc}
\DeclareMathOperator{\KL}{KL}
\DeclareMathOperator{\Gr}{Gr}

\DeclareMathOperator{\rk}{rk}
\DeclareMathOperator{\tg}{tg}

\newcommand{\vir}{\text{\rm vir}}

\newcommand{\ZZ}{\mathbb{Z}}
\newcommand{\LL}{{\mathbb L}}

\newcommand{\QQ}{{\mathbb Q}}

\newcommand{\bP}{{\mathbf P}}

\begin{document}

\title{Note on MacPherson's local  Euler obstruction}

\author[Jiang]{Yunfeng Jiang}
\address{Department of Mathematics\\
University of Kansas\\
405 Snow Hall\\
1460 Jayhawk Blvd\\
Lawrence 66045 
\\USA}
\email{y.jiang@ku.edu}

\thanks{}
  


\date{}

\begin{abstract}
This is a note on MacPherson's local Euler obstruction, which plays an important role recently in Donaldson-Thomas theory by the work of Behrend. 

We introduce MacPherson's original definition, and prove that it is equivalent to the algebraic definition used by Behrend, following the method of Gonzalez-Sprinberg.  We also give a formula of the local Euler obstruction in terms of Lagrangian intersections.  As an application, we consider a scheme or DM stack $X$ admitting a symmetric obstruction theory.  Furthermore we assume that there is  a $\CC^*$ action on $X$, which makes the obstruction theory $\CC^*$-equivariant. The $\CC^*$-action on the obstruction theory naturally gives rise to a cosection map in the sense of Kiem-Li. We prove that Behrend's weighted Euler characteristic of $X$ is the same as Kiem-Li localized invariant of $X$ by the $\CC^*$-action. 
\end{abstract}

\maketitle

\section{Introduction}
 
\subsection{Local Euler obstruction}
Local Euler obstruction of a scheme or a variety was originally introduced by MacPherson \cite{MacPherson} to study Chern class of singular algebraic varieties.  The definition of the local  Euler obstruction is given by topological and geometrical method in Section 3 of \cite{MacPherson}, by using the Nash blow-up.   In \cite{Sprinberg}, Gonzalez-Sprinberg
proves an algebraic formula using Segre class of normal cone of a closed subscheme inside a scheme. 

Let $Z$ be a prime cycle in a scheme  or DM stack $X$, the local Euler obstruction $\eu(Z)$ is a constructible function on $X$, and all the local Euler obstructions form a basis in the group $F(X)$ of constructible functions on $X$.  Furthermore, MacPherson proves that there is a unique transformation functor from the functor of constructible functions to the functor of homology groups satisfying the pushforward property, such that the homology class of the constant function $\mathds{1}_{X}$ for a smooth scheme $X$ is $c(X)\cap [X]$. 

A global characteristic class for the local Euler obstruction $\eu(Z)$ is the Chern-Mather class $c^{M}(Z)$ also introduced in Section 2 of \cite{MacPherson}.  This class is the pushforward $\nu_*(c(TZ)\cap [\widehat{Z}])\in A_*(Z)$ where $\nu: \widehat{Z}\to Z$ is the Nash blow-up and $TZ$ is the Nash tangent bundle. Global index theorem of MacPherson says that 
$$\int_{X}c^M(Z)=\chi(X, \eu(Z)),$$
if the scheme $X$ is proper, 
where $\chi(X, \eu(Z))$ is the weighted Euler characteristic weighted by the constructible function $\eu(Z)$. 
In the case that $Z=X$ and $X$ is smooth, Chern-Mather class $c^M(Z)$ is just the Chern class and $\chi(X, \eu(Z))$ is just the topological Euler characteristic of $X$. This is the Gauss-Bonnet theorem. 

In this paper we review the definition of  MacPherson on Euler obstruction   using topological method, and we prove that it is equivalent to the algebraic formula in the Theoreme of \cite{Sprinberg} following the arguments of Gonzalez-Sprinberg.  Note that \cite{Sprinberg} was written in French, and the translation proof may be helpful for  English  readers.

\subsection{Donaldson-Thomas theory}

The local Euler obstruction recently becomes very important in the study of Donaldson-Thomas theory by the work of Behrend \cite{Behrend}.  In \cite{Behrend}, Behrend defines a canonical cycle $\mathfrak{c}_X\in \mathcal{Z}_*(X)$ for a scheme or DM stack $X$, which Behrend calls the sign support of the intrinsic normal cone $\mathbf{c}_X$ in \cite{BF}. 
The Euler obstruction $\nu_X:=\eu(\mathfrak{c}_X)$ of this cycle is called the ``Behrend function" of $X$. 
The weighted Euler characteristic $\chi(X, \nu_X)$ is well defined. 

If the scheme or DM stack $X$ admits a symmetric obstruction theory $E_X$ introduced in \cite{Behrend}, then the virtual fundamental class 
$[X]^{\vir}\in A_0(X)$ is zero dimensional.  In the case that $X$ is proper, the virtual count (Donaldson-Thomas type invariant) $\#^{\vir}(X)$ is defined as
$$\int_{[X]^\vir}1.$$
Behrend proves that it is the same as the 
weighted Euler characteristic $\chi(X, \nu_X)$. 
The Donaldson-Thomas moduli space of stable coherent sheaves on Calabi-Yau threefolds admits a symmetric obstruction theory, and the Donaldson-Thomas invariant constructed in \cite{Thomas} is the weighted Euler characteristic weighted by the Behrend function of the moduli space. 

If a moduli space $X$ is not proper so that the integration does not makes sense, the weighted Euler characteristic $\chi(X, \nu_X)$ is defined as the Donaldson-Thomas invariant.  Examples include an amount of moduli spaces of quiver representations with potentials. 

\subsection{Lagrangian intersection}

Let $X\hookrightarrow M$ be an embedding into a smooth DM stack $M$. 
Then the symmetric obstruction theory $E_X$ on $X$ induces a cone $C$ inside $\Omega_M|_{X}$, which Behrend calls the obstruction cone. 
The key observation of \cite{Behrend} (suggested  by R. Thomas) is that the obstruction cone $C$ is a Lagrangian cone 
inside $\Omega_M$.  The obstruction cone $C$, \'etale locally on a chart of $X$, is given by the normal cone of a closed immersion. 

There is an isomorphism between the group of integral cycles $\mathcal{Z}_*(X)$ of $X$ and the group $\mathfrak{L}_X(\Omega_M)$ of conic Lagrangian cycles inside $\Omega_M$ supported on $X$.  The Lagrangian cone $C$ is a linear combination of conic Lagrangian cycles of the form $N_{Z/M}^*$, where $Z$ is a closed substack of $X$ and $N_{Z/M}^*$ is the closure inside $\Omega_M$ of the conormal bundle of smooth locus of $Z$.  Note that $N_{Z/M}^*$ gives all the irreducible components of the cone $C$, and its image under the morphism $\pi: C\to X$ gives the integral cycle $\mathfrak{c}_X$.  Ginzburg's theory \cite{Ginzburg} tells us that the Lagrangian intersection $I(N_{Z/M}^*, [M])$ of $N_{Z/M}^*$ with the zero section is the weighted Euler characteristic $\chi(X, \eu(Z))$ of the Euler obstruction $\eu(Z)$ up to a sign when $Z$ is proper. 
On the other hand, the Lagrangian intersection also gives the degree zero Chern-Mather class of $Z$. 
So the Lagrangian intersection $I(C, [M])$ gives the weighted Euler characteristic $\chi(X, \nu_X)$, hence the Donaldson-Thomas invariant $\#^{\vir}(X)$.

The Lagrangian intersection $I(C, [M])$ can be explained using another (real) analytic section of $\Omega_M$ as studied in Theorem 9.5.3 and Theorem 9.7.11 of \cite{KS}, and the Appendix in \cite{KL}.   By fixing a Hermitian metric on $\Omega_M$, we may choose a small perturbation $\Gamma$ of the zero section of $\Omega_M$. Then $I(C, [M])=I(C, \Gamma)$,  see Theorem \ref{main_theorem_1}.

\subsection{Applications}

The idea of Lagrangian intersection above has applications in Donaldson-Thomas theory if the scheme (the moduli space) $X$ admits a 
$\CC^*$ action.  Let $F\subset X$ be the fixed point locus of the $\CC^*$ action. Then 
the weighted Euler characteristic 
$$\chi(X,\nu_X)=\chi(F, \nu_X|_{F}),$$ 
where $\nu_{X}|_{F}$ is the restriction of $\nu_X$ to $F$. 
The $\CC^*$ action on $X$ naturally gives rise to a cosection morphism 
$\sigma: \Omega_{X}\to \cO_X$ in the sense of \cite{KL} such that the degenerate locus is exactly the fixed point locus 
$F$.  The virtual dimension of $X$ with symmetric obstruction theory $E_X$ has dimension zero.  Hence the Kiem-Li localized 
virtual cycle $[F]^{\vir}_{\loc, \KL}\in A_0(F)$ is a dimension zero class.  We prove that the Kiem-Li localized invariant 
$$\int_{[F]^{\vir}_{\loc, \KL}}1=\chi(F,\nu_X|_{F}),$$
see Theorem \ref{main_theorem_2}.
We should remark that in the case that $X$ is proper, this is a direct application of Theorem 4.18 of Behrend \cite{Behrend} and Theorem 1.1 of Kiem-Li \cite{KL}.  So the key point of this paper is to prove Theorem  \ref{main_theorem_2} in the case that $X$ is nonproper, but the $\CC^*$ fixed locus $F$ is proper. 
Since one can do $\CC^*$ localization on the Kiem-Li localized virtual cycle,  it is hoped that Kiem-Li calculation is easier than 
the calculation of the Behrend function. 
If the $\CC^*$ action only has isolated fixed points, we recover Theorem  3.4 of Behrend and Fantechi \cite{BF2}. 

This note is organized as follows.  Section \ref{local_Euler_obstruction} introduces MacPherson's original definition of local Euler obstruction.   In Section \ref{Algebraic_formula_local_Euler_obstruction} we prove the algebraic formula of local Euler obstruction following the method in \cite{Sprinberg}. 
In Section \ref{Global_index_Lagrangian_intersection}  the global index theorem and the Lagrangian intersection are discussed, and in Section \ref{Application} we give the application of Euler obstruction and Lagrangian intersection to Behrend theory and Kiem-Li cosection  localization.

\subsection*{Acknowledgement}

The author would like to thank Professors Kai Behrend, Jun Li and Richard Thomas for valuable discussions.  The last section of this paper was motivated by a discussion with Kai Behrend when the author was visiting UBC in November 2012,  a discussion with Jun Li when the author was visiting Stanford in December 2012,  and the work with Richard Thomas about virtual signed Euler characteristics.  He thanks Jim Bryan, Kai Behrend (UBC), and Jun Li (Stanford) for hospitality. 
This work is partially supported by Simons Foundation Collaboration Grant 311837, and NSF Grant DMS-1600997.

\section{Local Euler obstruction of MacPherson}\label{local_Euler_obstruction}

\subsection{Nash blow-up}\label{Nash_blow_up}
Let $Z$ be a $d$-dimensional scheme or DM stack, which is embedded into a smooth scheme or DM stack $A$.  Let 
$\Gr_{d}(TA)$ be the Grassmannian bundle of $d$-dimensional subspaces of the bundle $TA$ over $A$.  There is a section
$$
\begin{array}{llll}
s: &Z^{\circ} & \rightarrow& \Gr_{d}(TA)\\ \\
&z &\mapsto  &T_z(Z^{\circ})
\end{array}
$$
of smooth locus $Z^{\circ}\subset Z$ into the Grassmannian bundle, where $T_z(Z^{\circ})$ is the tangent space of $Z^{\circ}$ at the smooth 
point $z$. Let $\widehat{Z}$ be the closure of the image $s$ inside the Grassmannian bundle. 
Denote by
$$\nu: \widehat{Z}\to Z$$
the map given by the restriction of the projection of $\Gr_{d}(TA)$.  There is a vector bundle $TZ$ over $\widehat{Z}$, which is the restriction of the tautological bundle over the Grassmannian $\Gr_{d}(TA)$.  The map $\nu: \widehat{Z}\to Z$ is called the 
\textbf{Nash blow-up} of $Z$. 

\begin{remark}
The construction of Nash blow-up of $Z$ is independent to the embedding. 
\end{remark}

\subsection{Local Euler obstruction}
We give MacPherson's original local Euler obstruction for a scheme or DM stack $Z$.  Let $Z\hookrightarrow A$ be an embedding. 

Let $P\in Z$. We choose a local coordinates $(z_1, \cdots, z_n)$ of $A$ at the point $P$, such that $z_i(P)=0$. 
Let $\|z\|^2=\sqrt{z_1\overline{z}_1+\cdots+z_1\overline{z}_n}$. Then $\|z\|^2$ is a real-valued function, and the differential $d\|z\|^2$
may be considered as a section of $T^*A$, where  $T^*A$ is the cotangent bundle of $A$ but viewed as a real vector bundle. The section $d\|z\|^2$ pulls back and restricts to a section $r$ of $T^*Z$, which is the dual of $TZ$. 

\begin{remark}
Here we use $T^*A$ to represent the real cotangent bundle of $A$. Later on in Section \ref{Application} we use $\Omega_A$ as the cotangent bundle of a scheme or DM stack $A$. 
\end{remark}

\begin{lemma}
For small enough $\epsilon>0$, the section $r$ is nonzero over $\nu^{-1}(z)$ where $0<\|z\|\leq \epsilon$. 
\end{lemma}
\begin{proof}
Let $\widehat{Z}_{P}=\nu^{-1}(P)$ be the fibre over $P$, and $K$ be the set of zeros of $r$.  Then 
$\widehat{Z}_{P}\subset K$.  Suppose by contradiction there exists $x\in \widehat{Z}_{P}\cap \overline{K-\widehat{Z}_{P}}$, where 
$\overline{K-\widehat{Z}_{P}}$ (closure of $K-\widehat{Z}_{P}$) is real sub-analytic. So by Bruhat-Whitney Lemma, there exists an analytic curve 
$\widetilde{C}: [0,t]\to \overline{K-\widehat{Z}_{P}}$ such that $\widetilde{C}(0)=x$ and $\widetilde{C}((0,t])\subset K-\widehat{Z}_{P}$.

Projecting on $Z$, an analytic curve $C$ passing through $P$ is obtained.  Let $\mathcal{S}$ be a Whitney stratification of $Z$. 
For $t\neq 0$ sufficiently small, $C((0,t])$ is fully contained in a stratum $S$.  Let $P^\prime\neq P$ be a point on the curve $C$. Then the right secant $PP^\prime$ is orthogonal to the limit tangent space $T_{P^\prime}$ at $P^\prime$ (i.e. a point of $\widehat{Z}_{P^\prime}\cap \widetilde{C}$ which projects to $P^\prime$).   However from Condition A of Whitney, $T_{P^\prime}$ contains the tangent space to the stratum at $P^\prime$, that is , it contains the right tangent space to $C$ at $P^\prime$.  So 
$PP^\prime$ is orthogonal to the tangent to $C$ at $P^\prime$ for all $P^\prime\in C$, and 
$P^\prime\neq P$, which is impossible. 
\end{proof}

Let $B_{\epsilon}$ be the $\epsilon$-ball $\{z|\|z\|\leq \epsilon\}$ and $S_{\epsilon}$ be the $\epsilon$-sphere 
$\{z|\|z\|= \epsilon\}$.  The obstruction to extending $r$ as a nonzero section of $T^*Z$ from $\nu^{-1}S_{\epsilon}$ to 
$\nu^{-1}B_{\epsilon}$, which MacPherson denoted it by $\eu(T^*Z, r)$, lies in 
$H^{2d}(\nu^{-1}B_{\epsilon}, \nu^{-1}S_{\epsilon}; \ZZ)$.  Let 
$\cO\in H_{2d}(\nu^{-1}B_{\epsilon}, \nu^{-1}S_{\epsilon}; \ZZ)$ be the orientation class. 

\begin{definition}
The local Euler obstruction of $Z$ at $P$ is defined as:
$$\eu(Z)(P)=\langle \eu(T^*Z, r), \cO\rangle.$$
\end{definition}

The local Euler obstruction has the following properties, see Section 3 in  \cite{MacPherson}:
\begin{enumerate}
\item $\eu(Z)(P)=1$ if $Z$ is nonsingular at $P$;
\item  If $Z$ is a curve, then $\eu(Z)(P)$ is the multiplicity of $Z$ at $P$. If $Z$ is the cone on a nonsingular plane curve of degree $d$ and $P$ is the vertex, $\eu(Z)(P)=2d-d^2$;
\item  $\eu(Z_1\times Z_2)(P_1\times P_2)=\eu(Z_1)(P_1)\cdot \eu(Z_2)(P_2)$;
\item If we have a scheme $Z$ or DM stack which is reducible at the point $P$, and let $Z_i$ are the irreducible components, then 
$\eu(Z)(P)=\sum\eu(Z_i)(P)$.
\end{enumerate}

The local Euler obstruction $\eu(Z)$ is a constructible function, which was proved in algebraic sense by Kennedy in Lemma 4 of \cite{Kennedy}. 
Let $F(Z)$ be the group of constructible functions with integer values if $Z$ is a scheme and rational values if $Z$ is a DM stack. 
Then there is a mapping 
\begin{equation}\label{Algebraic_cycle_constructible-function}
T:  \mathcal{Z}_*(Z)\to F(Z)
\end{equation}
given by:
$$\sum a_iZ_i\mapsto \sum a_i\eu(Z_i).$$
MacPherson proved that $T$ is an isomorphism of abelian groups in Lemma 2 of 
\cite{MacPherson}. The proof is true for DM stacks by taking $\QQ$ coefficients for the constructible function.

\section{Algebraic formula of local Euler obstruction}\label{Algebraic_formula_local_Euler_obstruction}

\subsection{Algebraic formula}

In this section we give the algebraic explanation of the  local Euler obstruction.  We mainly follow the proof of Gonzalez-Sprinberg \cite{Sprinberg}.

Let $Z$ be a prime cycle in a DM stack $X$.  Recall $\nu: \widehat{Z}\to Z$ is the Nash blow-up of $Z$. 
\begin{theorem}\label{Euler_Obstruction_Algebraic_formula}
We have:
$$\eu(Z)(P)=\int_{\nu^{-1}(P)}c(TZ)\cap s(\nu^{-1}(P), \widehat{Z}),$$
where $s(\nu^{-1}(P),\widehat{Z})$ is the Segre class of the normal cone of  $\nu^{-1}(P)$ in $\widehat{Z}$.
\end{theorem}
\begin{remark}
This formula is used by Behrend \cite{Behrend} in Section 1.2 of \cite{Behrend} as the definition of local Euler obstruction.
\end{remark}

\subsection{Proof of Theorem \ref{Euler_Obstruction_Algebraic_formula}}

The local Euler obstruction $\eu(Z)(P)$ is local, then we may suppose that $Z\subset \CC^n$ and take $P$ to be the origin. 
The Nash tangent bundle is $TZ$ and $\widehat{Z}_0=\nu^{-1}(0)$ is the fibre over $0\in Z\subset \CC^n$.  Consider the following diagram:
\begin{equation}
\xymatrix{
\widehat{Z}_0\ar[r]\ar[d]& \widehat{Z}\ar[r]\ar[d]^{\nu}& \Gr_d(T\CC^n)\ar[d]\\
0\ar[r]& Z\ar[r]& \CC^n
}
\end{equation}

\textbf{Step 1:}  Let $T\CC^n=\CC^n\times\CC^n$ be the trivial tangent bundle over $\CC^n$, and let 
$E:=\nu^{*}(T\CC^n|_{Z})$ be the pullback by $\nu$ to $\widehat{Z}$.  Then we have the following diagram:
$$
\xymatrix{
TZ\ar[r]\ar[rd]& E\ar[d]\\
&\widehat{Z}
}
$$
where the Nash tangent bundle $TZ$ can be taken as a subbundle of $E$.  There is a canonical section $\CC^n\to T\CC^n$, by 
$P\mapsto \overrightarrow{OP}$ (the vector of $P$), then through the map $\nu$, we have an induced section $\rho$ of $E$ 
over $\widehat{Z}$.  By choosing a Hermitian metric on $E$ (denoted it by a Hermitian form $s$),  then we can project $\rho$ to $TZ$ and get a section $\sigma_{s}$ of $TZ$ over $\widehat{Z}$. The Hermitian form $s$ induces an isomorphism $TZ\cong T^*Z$, which transforms the section $\sigma_s$ to $r$ we defined before.  Hence 
$\eu(T^*Z, r)\cong \eu(TZ, \sigma_s)$.

\begin{remark}
Note that the orientation of $T^*Z$ is the one from the identification between complex cotangent bundle and real cotangent bundle,  and the orientation of $TZ$ is the one coming from the complex structure. 
\end{remark}

\textbf{Step 2:} Let $\epsilon>0$ be sufficiently small, such that $\sigma_s$ does not vanish on $\nu^{-1}S_{\epsilon}$. Set 
$$V=\nu^{-1}B_{\epsilon}, \quad  \partial V=\nu^{-1} S_{\epsilon}.$$
In order to calculate $\eu(TZ, \sigma_s)$, we consider the class $\omega$ of universal obstruction of $TZ$, determined by the canonical section of $p_{TZ}^*TZ$,  which is 
$$v\mapsto (v,v),$$
where $v\in TZ$, and $p_{TZ}: TZ\to \widehat{Z}$ is the projection.  The class
$$\omega\in H^{2d}(TZ, TZ- \widehat{Z})=H^{2d}_{\widehat{Z}}(TZ)$$
after identifying $\widehat{Z}$ as the zero section of $TZ$, and then it restricts to a class in 
$H^{2d}_{V}(TZ|_{V})$, still denoted by $\omega$. 

The section $\sigma_s$ of $TZ$ induces:
$$(V, V-\widehat{Z}_0)\to (TZ|_{V}, TZ|_{V}-V),$$
and induces:
$$\sigma_s^*: H^{2d}_{V}(TZ|_{V})\to H^{2d}_{\widehat{Z}_0}(V).$$
Let $\epsilon>0$ be small enough so that $V$ is a neighbourhood of $\widehat{Z}_0$ and $V-\widehat{Z}_0$ retracts to $\partial V$. 
So
$$H^{*}_{\widehat{Z}_0}(V)= H^{*}(V, V-\widehat{Z}_0)\cong H^{*}(V, \partial V).$$
Then 
$$\sigma_s^*(\omega)\in H^{2d}_{\widehat{Z}_0}(V)\cong H^{2d}(V, \partial V)$$
is the class $\eu(TZ, \sigma_s)$ by universal property of $\omega$.  So we have $H^{2d}_{\widehat{Z}_0}(V)\cong H^{2d}_{\widehat{Z}_0}(\widehat{Z})$, and $\sigma_s^*(\omega)$ can be considered in $H^{2d}_{\widehat{Z}_0}(\widehat{Z})$.
So
\begin{equation}
\eu(Z)(0)=\deg(\sigma_s^*(\omega)\cap [V,\partial V])=\deg(\sigma_s^*(\omega)\cap[\widehat{Z}]).
\end{equation}

\textbf{Step 3:}  The next step is to construct a scheme or variety associated with $\widehat{Z}$ such that we have an algebraic section of $TZ$ on it, since $\sigma_s$ is analytic.  Our method is to split the bundle $E$ as the sum of $TZ$ and its complement without the aid of Hermitian form, and do it algebraically.  More precisely, let us consider the exact sequence of vector bundles on $\widehat{Z}$:
$$0\to TZ\longrightarrow E\stackrel{j}{\longrightarrow} Q\to 0,$$
where $Q=E/TZ$, and $\rk(TZ)=d$, $\rk(E)=n$.

Let $\Gr_{n-d}(E)$ be the Grassmannian bundle over $\widehat{Z}$, whose fibre over $x\in \widehat{Z}$ is the space of $(n-d)$-dimensional subspaces in $E_x$.  Consider the open subset $U\subset \Gr_{n-d}(E)$, whose fibre over $x\in\widehat{Z}$ is the complement of $TZ|_{x}$ of $E_x$, i.e. a subspace $a: Q\subset E$ such that $j\circ a=id_{Q}$. 
Let $p: U\to \widehat{Z}$ be the projection, whose fibre over $x\in \widehat{Z}$ is an affine space of dimension $d(n-d)$. This $U$ is a principal homogeneous space.  On $U$ consider the pullback bundle $p^*E$, its section $\rho$ and $\sigma_s$ of $TU=p^*TZ$ (we use the same notations). 

Let $S$ be the complement of $TU$, obtained by restriction on $U$ of the tautological fibre of the Grassmannian. Let $\sigma$ (reap. $\eta$) be the projections of $\rho$ on $TU$ (resp. $S$). Then 
$$E=TU\oplus S,$$ and
$$\rho=\sigma+\eta.$$
Let $U_0=p^{-1}(\widehat{Z}_0)$. Then we have the following diagram:
$$
\xymatrix{
U_0\ar[r]\ar[d]& U\ar[d]\\
\widehat{Z}_0\ar[r] &\widehat{Z}
}
$$

\begin{proposition}
Let $u\in U_0$, and $\epsilon>0$ sufficiently small. Then there exists a neighbourhood $W$ of $u$ in $U$ such that for all $w\in W$, 
$$\|\eta(w)\|<\epsilon \|\sigma(w)\|.$$
\end{proposition}
\begin{proof}
A point $u\in U_0$ represents a point in $\widehat{Z}_0$, i.e. a limit  tangent space $T$ at the origin, and an additional space $S$ in 
$E$.  Fix a Hermitian form $s$, then it gives the decomposition of $E$,  i.e. $S=T^{\perp}$.  Then $s$ induces a section of $U$ passing through $u$.  It suffices to prove the proposition on this section, because $U$ is a principal homogeneous space, and we can identify $\widehat{Z}$ with this section. So we shall be induced to prove on $\widehat{Z}$ the following statement:  Let $x\in \widehat{Z}_0$, $\epsilon>0$, then there exists a neighbourhood $\widehat{V}$ of $x$ in $\widehat{Z}$, such that for all $v\in \widehat{V}$, we have
\begin{equation}\label{inequality}
\|\eta_{s}(v)\|<\epsilon \|\sigma_s(v)\|
\end{equation}
where $\sigma_s$ and $\eta_s$ are the components of $\rho$ determined by $s$. 

Let $\{M_{\alpha}\}$ be a Whitney stratification of $Z$.  So it suffices to show that there exists a neighbourhood $\overline{V}$ of $0$ in $Z$ such that for all $\alpha$ and for all $z\in M_{\alpha}\cap\overline{V}$, one has $\mbox{tg}^2(oz, T_{M_{\alpha},z})<\epsilon^2$, where $T_{M_{\alpha},z}$ is the tangent space of $M_{\alpha}$ at the point $z$.  
(This is because $|\mbox{tg}(oz, T_{Z})|=\frac{\|\eta_s(v)\|}{\|\sigma_s(v)\|}$.)
Here $\tg(oz, T_{M_{\alpha},z})$ and  $\tg(oz, T_{Z})$ are the trigonometric tangent functions of the angles between $oz$ and the tangent planes. 

We prove this by contradiction.  Suppose that there is an $\alpha$ such that for every neighbourhood $\overline{V}$ of $0$ there exists 
$z\in M_{\alpha}\cap \overline{V}$ and 
$$\mbox{tg}^2(oz, T_{M_{\alpha, z}})\geq \epsilon^2.$$
Consider all real semi-analytic sets
$$R=\{z\in M_{\alpha}| \mbox{tg}^2(oz, T_{M_{\alpha, z}})\geq \epsilon^2\}.$$
Then $0\in \overline{R}$, the closure of $R$. By the Bruhat-Whitney Lemma, there is a real analytic curve
$$C: [0,t]\to \overline{R},$$
such that
$C(0)=0$, and $C((0,t])\subset R$. So this curve would have the following property: at any point $P\neq 0$, $P\in C$, 
$\mbox{tg}^2(OP, \mbox{tg}_P)\geq \epsilon^2$. This is impossible because the $\mbox{tg}$ tends to zero when $P\to 0$.

As a result we have (\ref{inequality}). In fact, giving  $v\in \widehat{Z}$ is equivalent to giving $z=\nu(v)\in Z$ and a limit tangent space 
$T_z$ of the tangent spaces to smooth points in a neighbourhood of $z$.  So by Whitney condition A, we have:
$T_{M_{\alpha},z}\subset T_z$ and 
$$|\mbox{tg}(oz, T_{M_{\alpha},z})|\geq |\mbox{tg}(oz, T_z)|=\frac{\|\eta_s(v)\|}{\|\sigma_s(v)\|}.$$
\end{proof}

\begin{corollary}\label{Cor_inequality}
The set $U_0$ is open and closed in the set $K$ of zeroes of $\sigma$. 
\end{corollary}
\begin{proof}
Indeed, $U_0$ is the set of zeros of $\rho$, and $\rho=0$ implies $\sigma=0$.  We show $U_0$ is open by contradiction. 
Suppose there exists a sequence $\{u_i\}$ contained in $K-U_0$ which converges to $u\in U_0$. By the previous proposition, 
for $i$ sufficiently large, $\sigma(u_i)=0$ implies $\eta(u_i)=0$. Since $\rho=\sigma+\eta$, $\rho(u_i)=0$, that is to say, $u_i\in U_0$ for $i$ large enough and we have a contradiction. 
\end{proof}

\textbf{Step 4:}  We continue the proof of the theorem.  Let 
$K^\prime=K-U_0$ be 
the union of connected components of $K$ (the zeros of $\sigma$) disjoint to $U_0$. Let 
$$W:=U-K^\prime;$$
which is an open neighbourhood of $U_0$ in $U$, according to Corollary \ref{Cor_inequality}.  Note also the restriction to $W$ of $p: U\to \widehat{Z}$ and consider the following diagram:
$$
\xymatrix{
U_0\ar[r]\ar[d]&W_{V}\ar[r]\ar[d]&W\ar[d]_{p}& p^*TZ\ar[l]\ar[d]^{\pi}\\
\widehat{Z}_0\ar[r]& V\ar[r]&\widehat{Z}& TZ\ar[l]
}
$$
where $V=\nu^{-1}B_{\epsilon}$ and $W_{V}=p^{-1}(V)$. 
Then we have: 
$$(W_{V}, W_{V}-U_0)\stackrel{p}{\rightarrow}(V, V-\widehat{Z}_0)\stackrel{\sigma}{\rightarrow}(TZ, TZ-\widehat{Z})$$
and 
$$p^*\sigma_s^*(\omega)\in H^{2d}_{U_0}(W_{V})\cong H^{2d}_{U_0}(W).$$
So:
$$p^*(\sigma_s^*(\omega)\cap [\widehat{Z}])=p^*\sigma_s^*\omega \cap [W]\in H_{2d(n-d)}(U_0).$$
In addition, we have morphisms of pairs of spaces:
$$(W, W-U_0)\stackrel{\sigma}{\rightarrow}(p^*TZ|_{W}, p^*TZ|_{W}-W)\stackrel{\pi}{\rightarrow}(TZ, TZ-\widehat{Z})$$
where $\sigma^*\pi^*\omega\in H^{2d}_{U_0}(W)$. 

We prove that $p^*\sigma_s^*\omega=\sigma^*\pi^*\omega$. Note that 
$s: \widehat{Z}\to U$ is the section induced from the Hermitian form and $s|_{V}$ is the restriction to $V$.  It  induces the isomorphism
$(s|_{V})^*: H^{2d}_{U_0}(W)\stackrel{\cong}{\rightarrow}H^{2d}_{\widehat{Z}_0}(V)$, because we have the following commutative diagram:
$$
\xymatrix{
H^{2d}_{U_0}(W)\ar[r]^{\cong}\ar[d]_{(s|_{V})^*}& H^{2d}_{\widehat{Z}_0}(U)\ar[d]_{s^*}\\
H^{2d}_{\widehat{Z}_0}(V)\ar[r]^{\cong}& H^{2d}_{\widehat{Z}_0}(\widehat{Z})
}
$$
where the horizontal arrows are excisions. 

In fact, $(s|_{V})^*$ is the inverse of $p^*$, so 
$p^*\sigma_s^*\omega=\sigma^*\pi^*\omega$ is equivalent to $\sigma_s^*\omega=(s|_{V})^*\sigma^*\pi^*\omega$.  
The latter is true by the following commutative diagram:
$$
\xymatrix{
W\ar[r]^{\sigma}& p^*TZ\ar[d]^{\pi}\\
V\ar[r]^{\sigma_s}\ar[u]^{s|_{V}}& TZ
}
$$
As a consequence we have:
\begin{equation}
p^*\sigma_s^*\omega\cap [W]=\sigma^*\pi^*\omega\cap [W]
\end{equation}
and by projection:
\begin{equation}
\sigma_*(\sigma^*\pi^*\omega\cap [W])=\pi^*\omega\cap [\sigma(W)].
\end{equation}
The latter class is the product of the fundamental class of $W$ with fundamental homology class of $\sigma(W)$, so this is the class of intersection cycle $[W\cdot_{\sigma}W]$ in $A_*(W\cap \sigma(W))$. So 
\begin{equation}
\pi^*\omega\cap [\sigma(W)]=[W\cdot_{\sigma}W].
\end{equation}

\textbf{Step 5:}  There is an algebraic formula to calculate the above intersection cycle due to Fulton \cite{Fulton}. 
The variety $W$ is embedded as zero section into $p^*TZ$, which is locally complete intersection. 
The scheme or DM stack intersection $W\cap \sigma(W)$ is defined by $\sigma=0$ in $W$, which we will denote it by 
$W_0$.  The dimension of $W$ is $d+d(n-d)$ and its codimension in $p^*TZ$ is $d$. Look at the following diagram:
$$
\xymatrix{
W\ar[r]^{\sigma} & p^*TZ\\
W_0\ar[r]\ar[u]& W\ar[u].
}
$$

Then 
\begin{equation}
[W\cdot_{\sigma}W]=\left(c(p^*TZ|_{W_0})\cap s(W_0, W)\right)_{d(n-d)}
\end{equation}
i.e. the component of dimension $d(n-d)$ in $A_*(W_0)$ is the cap product of the Chern class of $p^*TZ|_{W_0}$ with the Segre class 
of $W_0$ in $W$, see Chapter 6 in  \cite{Fulton}. 

\begin{corollary}
With the above notation,  we have
$$s(W_0, W)=s(U_0, U).$$
\end{corollary}
\begin{proof}
Let $\mathcal{I}=(\sigma_1, \cdots, \sigma_n)$ (resp. $\mathcal{J}=(\rho_1, \cdots, \rho_n)$)  be the ideal of definition of 
$W_0$ (reap. $U_0$), where $\sigma_i$ (resp. $\rho_i$) is the $i$-th component of $\sigma$ (resp. $\rho$). 

Let $f: \overline{W}\to W$ be the blow up of $W$ along $W_0$.  Let $\overline{\mathcal{I}}=f^*\mathcal{I}=(\overline{\sigma}_1, \cdots, \overline{\sigma}_n)$  (resp. $\overline{\mathcal{J}}=f^*\mathcal{J}=(\overline{\rho}_1, \cdots, \overline{\rho}_n)$), where 
$\overline{\sigma}_i=f^*\sigma_i$, (resp. $\overline{\rho}_i=f^*\rho_i$). 

We show that $\overline{\mathcal{I}}=\overline{\mathcal{J}}$.  On one hand, we have: 
$\mathcal{I}\subset \mathcal{J}$, because $\rho=0$ implies $\sigma=0$. So we have:
$\overline{\mathcal{I}}\subset \overline{\mathcal{J}}$. On the other hand, since $\rho=\sigma+\eta$ and $\|\eta\|<\epsilon \|\sigma\|$ in a neighbourhood of $U_0$, then $|\rho_i|<(1+\epsilon)\sum_{j=1}^{n}|\sigma_j|$ for $i=1, \cdots, n$.  Let $g$ be a local generator of $\overline{\mathcal{I}}$ (which is invertible). Then $\frac{\rho_i}{g}$ is a function locally bounded for all $i$, so holomorphic because $\overline{W}$ is normal. As a result, we have $\overline{\mathcal{J}}\subset \overline{\mathcal{I}}$. So $\overline{\mathcal{J}}=\overline{\mathcal{I}}$.

By invariance of Segre class under birational map, $s(U_0, W)=f_*s(\overline{U}_0, \overline{W})=s(W_0, W)$, where 
$\overline{U}_0=\overline{W}_0$. Finally, since $U_0$ is closed in $U$, we have $s(U_0, W)=s(U_0, U)$. 
\end{proof}

So 
\begin{equation}
[W\cdot_{\sigma}W]=\left(c(p^*TZ|_{U_0})\cap s(U_0, U)\right)_{d(n-d)}
\end{equation}
The projection $p$ is flat, $p^*s(\widehat{Z}_0, Z)=s(U_0, U)$. So 
\begin{equation}
[W\cdot_{\sigma}W]=\left(c(p^*TZ|_{\widehat{Z}_0})\cap s(\widehat{Z}_0, \widehat{Z})\right)_{0}
\end{equation}
So: 
$$\eu(Z)(P)=\deg(c(TZ|_{\widehat{Z}_0})\cap s(\widehat{Z}_0, \widehat{Z})).$$
This proves Theorem \ref{Euler_Obstruction_Algebraic_formula}.

\subsection{Another formula}

We make the following diagram:
\[
\xymatrix{
D\ar[r]\ar[d]&\widetilde{Z}\ar[d]^{b}\\
\widehat{Z}_0\ar[r]\ar[d]& \widehat{Z}\ar[d]^{\nu}\\
0\ar[r]&Z
}
\]
where $b: \widetilde{Z}\to \widehat{Z}$ is the blow-up of $\widehat{Z}$ along $\widehat{Z}_0$, and $D=b^{-1}(\widehat{Z}_0)$ is the exceptional divisor.  Let $\xi$ be the normal bundle of $D\subset \widetilde{Z}$. 

\begin{corollary}\label{formula_blow-up_Euler_obstruction}
$$\eu(Z)(P)=\int_{D}c_{d-1}(TZ-\xi)\cap [D],$$
where $c(TZ-\xi)=\frac{c(TZ)}{c(\xi)}$. 
\end{corollary}
\begin{proof}
This is from the definition of Segre class:
$$s(\widehat{Z}_0, \widehat{Z})=b_*(c(\xi)^{-1}\cap [D]).$$
So 
\begin{align*}
\deg(c(TZ)\cap s(\widehat{Z}_0, \widehat{Z}))&=\deg(c(TZ)\cap b_*(c(\xi)^{-1}\cap [D]))\\
&=\deg(c(TZ)\cup c(\xi)^{-1}\cap [D])\\
&=\deg(c_{d-1}(TZ-\xi)\cap [D]).
\end{align*}
\end{proof}

\begin{Example}
Let $Z=\spec\CC[x,y, u, v]/(yu-xv)$. Then $Z$ is smooth except the origin $P$, which is called conifold singularity. 
The natural embedding $Z\subset \CC^4$ gives the Nash blow-up
$$\nu: \widehat{Z}\to Z$$
which is $\cO_{\bP^1\times \bP^1}(-1,-1)$.  The Nash cotangent bundle is
$T^*Z=\bP T^*\CC^4/\cO(1)$. 

Note that the line bundle $\cO(1)$, when restrict to the exceptional divisor $\bP^1\times\bP^1$, is 
$\cO(1,1)$. 
Let $c_{1}(\cO(1,1))=x+y$. 
Then from Corollary \ref{formula_blow-up_Euler_obstruction}
\begin{align*}
\eu(Z)(P)&=\deg(c(TZ)\cap s(\bP^1\times\bP^1, \widehat{Z}))\\
&=\deg(\frac{1}{1+x+y}\cdot c(\cO(-1,-1))^{-1}\cap [\bP^1\times\bP^1])\\
&=2.
\end{align*}
\end{Example}

\section{Global index theorem and Lagrangian Intersection}\label{Global_index_Lagrangian_intersection}

\subsection{Chern-Mather class}

Fix a DM stack $X$ throughout this section. 
Let $Z\subset X$ be a prime cycle.  The Nash blow-up $\nu: \widehat{Z}\to Z$ and the Nash tangent bundle $TZ$ are defined in
Section \ref{Nash_blow_up}. 

\begin{definition}
The Chern-Mather class $c^{M}(Z)$ is:
$$c^{M}([Z])=\nu_{*}(c(TZ)\cap [\widehat{Z}]),$$
where $c^{M}([Z])\in A_*(X)$. 
\end{definition}

Let  $c_0^{M}(Z)$ be the degree zero part $c_0^{M}: \mathcal{Z}_*(X)\to A_{0}(X)$. 
The local Euler obstruction $\eu(Z)$ is a constructible function $X\to \ZZ$. The weighted Euler characteristic $\chi(X, \eu(Z))$ is given by
$$\sum_in_i\cdot \chi(\eu(Z)^{-1}(i)).$$

In Proposition 1.12 of \cite{Behrend}, Behrend proves 
\begin{theorem}
$$\int_{X}c^{M}(Z)=\chi(X, \eu(Z)),$$
if $X$ is  a proper scheme or global finite group quotient stack, or a grebe over a scheme. 
\end{theorem}
\begin{proof}
The scheme case is MacPherson's index theorem in \cite{MacPherson}.  Other cases were proved by Behrend \cite{Behrend} using properties of the local Euler obstruction. 
\end{proof}

In (\cite{Jiang}), the result is proved for proper DM stacks:
\begin{theorem}
$$\int_{X}c^{M}(Z)=\chi(X, \eu(Z)),$$
if $X$ is  a proper DM stack. 
\end{theorem}
\begin{proof}
In  \cite{Jiang}, the author defines ProChow group and classes for not necessarily proper DM stacks, generalizing the result in \cite{Aluffi}. 
There is a natural transformation functor from the group of constructible functions on DM stacks to the ProChow groups. 
Every ProChow group class of a DM stack $X$ gives a degree, which is the weighted Euler characteristic of the corresponding constructible function. 
In the case $X$ is proper, the ProChow group of $X$ coincides with the Chow group $A_*(X)$, and the Euler obstruction $\eu(Z)$ of the prime cycle $Z$ in $X$ is a constructible function, hence the degree of the  corresponding Chow group class (Chern-Mather class) gives its weighted Euler characteristic.  The formula in terms of Lagrangian intersection for orbifolds was addressed by Maulik and Treumann in \cite{MT}.
\end{proof}

\subsection{Lagrangian Intersection}

The degree zero Chern-Mather class $c_0^{M}(Z)$ and the weighted Euler characteristic  can be interpreted by Lagrangian intersections. We mainly follow Behrend's proof in Section 4.1 of \cite{Behrend}, only the last part involving small perturbation of zero section of vector bundles is not included in \cite{Behrend}. 

Fix an embedding $X\to M$ of the DM stack $X$ into a smooth DM stack $M$. We will explain the following diagram due to Behrend in 
Diagram (2) of \cite{Behrend}.
\begin{equation}\label{Key_Diagram}
\xymatrix{
\mathcal{Z}_*(X)\ar[r]^{\eu}_{\cong}\ar[dr]_{c^{M}_0}&F(X)\ar[r]^{\mbox{Ch}}_{\cong}\ar[d]^{c_0^{SM}}&\mathfrak{L}_{X}(\Omega_{M})\ar[dl]^{I(\cdot, [M])}\\
&A_0(X)&
}
\end{equation}
where $\mathcal{Z}_*(X)$ is the group of integral cycles of $X$, $F(X)$ is the group of constructible functions on $X$, and $\mathcal{L}_X(\Omega_M)$ is the subgroup of $\mathcal{Z}_n(\Omega_M)$ generated by the conic Lagrangian prime cycles supported on $X$. 
The maps $c^{M}_0$, $c_0^{SM}$ and $I(\cdot, [M])$ are degree zero Chern-Mather class, degree zero Chern-Schwartz-Mather class and the Lagrangian intersection with zero section of $\Omega_{M}$, respectively.  Note that in \cite{Behrend}, the notation of Lagrangian intersection with zero section is denoted by $0^{!}_{\Omega_{M}}(\cdot)$. 

We briefly explain the horizontal morphisms in the diagram.  The first map is the local Euler obstruction $\eu$ and it gives an isomorphism from $\mathcal{Z}_*(X)$ to $F(X)$, which is given by (\ref{Algebraic_cycle_constructible-function}). 

 Behrend  (Section 4.1, \cite{Behrend}) defined the following isomorphism of groups:
\begin{equation}\label{cycle_Lagrangian}
L:  \mathcal{Z}_{*}(X) \to \mathfrak{L}_{X}(\Omega_{M})
\end{equation}
which is given by
$$Z\mapsto (-1)^{\dim(Z)}N^*_{Z/M},$$
where $N^*_{Z/M}$ is the closure of the conormal bundle of smooth part of $Z$ inside $M$. 
Conversely there is an isomorphism:
\begin{equation}\label{Lagrangian_cycle}
\pi:  \mathfrak{L}_{X}(\Omega_{M})\to \mathcal{Z}_{*}(X) 
\end{equation}
which is given by
$$V\mapsto (-1)^{\dim(\pi(V))}\pi(V),$$
where $\pi: V\to X$ is the projection.  
Then the morphism $\mbox{Ch}$ is defined by the isomorphism $\eu$ and the morphism $L$ defined above. 
\\

\textbf{Proof of  Diagram \ref{Key_Diagram}:}
Let $Z\subset X$ be a prime cycle.  We may take the Nash blow-up $\nu: \widehat{Z}\to Z$ by embedding $Z\subset M$ into the smooth DM stack $M$. Let 
$$\mu: \widetilde{M}:=\Gr_d(\Omega_M)\to M$$
be the Grassmannian of rank $d$ quotient of $\Omega_M$. 
The Nash blow-up $\widehat{Z}$ is the closure inside $\widetilde{M}$ of  the section $Z\hookrightarrow \widetilde{M}$ for the smooth locus of $Z$.  Then there is an exact sequence of vector bundles:
$$0\to N|_{\widehat{Z}}\longrightarrow \mu^*\Omega_{M}|_{\widehat{Z}}\longrightarrow T^*Z\to 0$$
where $N|_{\widehat{Z}}$ is the kernel of the surjective map of the right arrow.  We have the following commutative diagram:
$$
\xymatrix
{
\widehat{Z}\ar[r]\ar[d]& N|_{\widehat{Z}}\ar[d]\\
[M]\ar[r]& \mu^*\Omega_{M}|_{\widehat{Z}}.
}
$$
So from Fulton Chapter 6 of  \cite{Fulton}, 
\begin{align*}
I(N|_{\widehat{Z}}, [M])&=c(\Omega_{M}|_{\widehat{Z}})\cdot s(\widehat{Z}, N|_{\widehat{Z}})\\
&=c(\Omega_{M}|_{\widehat{Z}})\cdot c(N|_{\widehat{Z}})^{-1}\cdot [\widehat{Z}]\\
&=c(T^*Z)\cdot [\widehat{Z}].
\end{align*}

From Proposition 4.6 in \cite{Behrend}, we have diagram:
\begin{equation}\label{Behrend_key_diagram}
\xymatrix{
\widehat{Z}\ar[r]^{\nu}\ar[d]_{0}& Z\ar[r]\ar[d]_{0}& M\ar[d]\\
N|_{\widehat{Z}}\ar[r]^{\eta}& N_{Z/M}^*\ar[r]& \Omega_{M}
}
\end{equation}
where $\eta: N|_{\widehat{Z}}\to N_{Z/M}^*$ is a proper birational map of integral stacks, where $N_{Z/M}^*$ is the closure inside 
$\Omega_{M}$ of the conormal bundle of the smooth part of $Z$.  Hence
\begin{align*}
c_0^{M}(Z)&=(-1)^{\dim(Z)}\nu_{*}(c(T^*Z\cap [\widehat{Z}]))\\
&=(-1)^{\dim(Z)}\nu_{*}(I(N|_{\widehat{Z}}, [M]))\\
&=(-1)^{\dim(Z)}I(\eta_{*}N|_{\widehat{Z}}, [M])\\
&=(-1)^{\dim(Z)}I(N_{Z/M}^*, [M]).
\end{align*}
$\square$

\subsection{Analytic version of the Lagrangian intersection}
In practice, it is often useful to do analytic or real version of Lagrangian intersection as in Theorem 9.5.3 and Theorem 9.7.11 of  \cite{KS}. 

We choose a Hermitian metric for the cotangent bundle $\Omega_{M}$ such that it is naturally isomorphic to the tangent bundle $TM$.  Let 
$$\psi: M\to\RR$$
be a real  $C^2$-function. Then the graph 
$\Gamma=d\psi$ is a real Lagrangian cycle in $\Omega_M$.  Suppose that 
\begin{enumerate}
\item $\{x\in M: \psi(x)\leq t\}$ is compact for every $t$;
\item $\Gamma\cap N_{Z/M}^*$ is compact. 
\end{enumerate}
Then 

\begin{theorem}\label{main_theorem_1}
The Lagrangian intersection:
$$I(N_{Z/M}^*, \Gamma)=I(N_{Z/M}^*, [M]).$$
\end{theorem}
\begin{proof}
 As a cycle, the cycle $N_{Z/M}^*$ is Lagrangian.  So
 from Theorem 9.5.3 and 9.7.11 in \cite{KS},
$$I(\Gamma, N_{Z/M}^*)=I(N_{Z/M}^*, [M]).$$ 
\end{proof}

\section{Application to Donaldson-Thomas invariants}\label{Application}

\subsection{Symmetric  obstruction theory}\label{symmetric_obstruction_theory}

The notion of symmetric perfect obstruction theory was defined by Behrend in \cite{Behrend}, and its main application is to Donaldson-Thomas theory.   The Donaldson-Thomas type moduli space of stable simple complexes over a smooth Calabi-Yau threefold $Y$ admits a symmetry obstruction theory. 

Let $X$ be a DM stack and $E_X$ an object in  the derived category $D^b(\Coh(X))$ of coherent sheaves on $X$ with amplitude contained in $[-1,0]$.
From \cite{BF} and \cite{LT},   $E_X$ is a  perfect obstruction theory on $X$ if there  exists a morphism 
$$\varphi: E_X\to \LL_{X}$$ 
in the derived category of coherent sheaves on $X$ such that 
$h^0(\varphi)$ is an isomorphism and $h^{-1}(\varphi)$ is surjective.
Here $\LL_X$ is the truncated cotangent complex of $X$. 
From \cite{BF}, the intrinsic normal cone $\cc_{X}$ is a subcone stack inside the 
cone stack $h^1/h^0(\LL_X^{\vee})$.

We assume that there exists a global resolution $E=[E^{-1}\stackrel{\phi}{\rightarrow} E^0]$ of $E_X$, which means there is a quasi-isomorphism
$$E\to E_X$$
in the derived category $D^b(\Coh(X))$ of coherent sheaves on $X$ and $E^i$ are vector bundles for $i=-1,0$.
Let 
$$E^{\vee}=[\phi^{\vee}: E_0\to E_1]$$
be the dual of of  $E$, where 
$E_0=(E^0)^{*}$, and $E_1=(E^{-1})^{*}$. 
The perfect obstruction theory gives  a closed immersion
of cone stacks $h^1/h^0(\LL_X^{\vee})\hookrightarrow h^1/h^0(E_X^\vee)=[E_1/E_0]$.
The latter Artin stack $h^1/h^0(E_X)^{\vee}=[E_1/E_0]$ is called the bundle stack and the intrinsic normal cone $\cc_X$ is a subcone stack in it.

The perfect obstruction theory $E_X$ is \textbf{symmetric} if there is a bilinear form
$$\Theta: E_X\stackrel{\cong}{\rightarrow} E_X^{\vee}[1]$$
which is non-degenerate and symmetric.  Hence this implies that $\rk E_X=\rk(E_X^{\vee}[1])=-\rk E_X^{\vee}=-\rk E_X$, so 
$\rk E_X=0$.
The obstruction sheaf $\ob=h^1(E_X^\vee)=h^0(E_X^{\vee}[1])=h^0(E_X)=\Omega_X$.

We fix an embedding $i: X\hookrightarrow M$ into a smooth DM stack $M$. Then
$$\Phi: \Omega_M|_{X}\to \ob=\Omega_X$$
 is an epimorphism of coherent sheaves.  Let $\mbox{cv}$ be the coarse moduli space of the intrinsic normal cone $\cc_X$. 
 Then
there is a 
Cartesian  diagram
\[
\xymatrix{
C\ar[r]^{}\ar[d]_{}& \Omega_M|_{X}\ar[d]^{}\\
\mbox{cv}\ar[r]^{}& \ob.
}
\]
The cone $C$ is called the obstruction cone of the symmetric perfect obstruction theory $E_X$.
\begin{proposition}(Theorem 4.9 in \cite{Behrend})
The cone $C\subset \Omega_{M}$ is conic Lagrangian. 
\end{proposition}

\begin{definition}
The virtual fundamental class $[X]^{\vir}$ is given by
$$[X]^{\vir}=0^{!}_{\Omega_{M}}([C])\in A_{0}(X).$$
The Donaldson-Thomas type invariant of $X$ is defined by
$$\#^{\vir}(X)=\int_{[X]^{\vir}}1$$
when $X$ is proper.
\end{definition}

\subsection{Weighted Euler characteristic}\label{weighted_Euler_char}

The DM stack $X$ has a canonical integral cycle $\mathfrak{c}_X\in \mathcal{Z}_*(X)$ in Section 1.1 of \cite{Behrend}. 
This cycle is defined as follows:  on a local chart $U\to X$, which is \'etale,  and an embedding $U\hookrightarrow M$, the cycle is 
$$\mathfrak{c}_X|_{U}=\sum_{i}(-1)^{\dim(\pi(C_i))}\cdot \mbox{mult}(C_i)\cdot \pi(C_i),$$
where the sum is over all irreducible components $C_i$ of  the normal cone $C_{U/M}$, and $\pi: C_{U/M}\to U$ is the projection. 
The set $\pi(C_i)$ is an irreducible closed subscheme or substack in $U$ and $\mbox{mult}(C_i)$ is the multiplicity of the component $C_i$ at the generic point.  Behrend proves that $\mathfrak{c}_{X}|_{U}$ is independent to the embedding and these local data glue to give the canonical integral cycle $\mathfrak{c}_X$. 

\begin{definition}
The Behrend function $\nu_X$ is defined by:
$$\nu_{X}=\eu(\mathfrak{c}_X),$$
which is the local Euler obstruction of the cycle $\mathfrak{c}_X$. 
\end{definition}

Look at Diagram 
(\ref{Key_Diagram}),  when applying the degree zero Chern-Mather class to the Behrend function we get 
$c^{M}_0(\nu_X)\in A_0(X)$ and also $\mbox{Ch}(\nu_X)=C$.  From the above section, the virtual fundamental class $[X]^{\vir}$ is the Lagrangian intersection of the cone 
$C$ with the zero section of $\Omega_M$.  Hence Behrend proves that 
\begin{equation}
[X]^{\vir}=c^{M}_0(\nu_X).
\end{equation}
So:
\begin{theorem}\label{Behrend}
If $X$ is proper and admits a symmetric obstruction theory $E_X$, then 
$$\chi(X, \nu_X)=\int_{[X]^\vir}1.$$
\end{theorem}
This was proved in Theorem 4.18 of \cite{Behrend}.

\subsection{$\CC^*$-equivariant   obstruction theory}

In this section we assume that there is a $\CC^*$ action on the scheme or DM stack $X$. 
We talk about  $\CC^*$-equivariant  obstruction theory on $X$. 
The general $G$-equivariant obstruction theory and $G$-equivariant symmetric obstruction theory  for an algebraic group $G$ has been discussed  by  Behrend and Fantechi in \cite{BF2}. 

Let $D^b(\Coh(X))^{\CC^*}$ be the derived category of $\CC^*$ equivariant coherent sheaves over $X$. 
From  Section 2.2 in \cite{BF2}, 
\begin{definition}
A $\CC^*$-\textbf{equivariant} perfect obstruction theory on $X$ is a morphism $E_X\to \LL_{X}$ in the category 
$D^b(\Coh(X))^{\CC^*}$.  As mentioned by Behrend, this is originally sue to Graber-Pandharipande. 

A  $\CC^*$-\textbf{symmetric equivariant} obstruction theory is given by $(E_X\to \LL_X, \Theta: E_X\to E_X^{\vee}[1])$ of morphisms in 
$D^b(\Coh(X))^{\CC^*}$, such that $E_X\to \LL_X$ is an equivariant perfect obstruction theory, and 
$\Theta: E_X\to E_X^{\vee}[1]$ is an isomorphism satisfying $\Theta^{\vee}[1]=\Theta$.
\end{definition}

We include an example of Behrend here so that an equivariant symmetric obstruction theory locally always has the following form.
\begin{Example}
In Section 3.4 of \cite{Behrend}, Behrend proves that a symmetric obstruction theory on $X$ locally is given by an almost closed one form. 
Let $\omega=\sum_{i=1}^{n}f_idx_i$ be an almost one form on $M\cong \aA^n$ so that $X=Z(\omega)$ is the zero locus of $\omega$. 
Then $H(\omega)=[T_M|_{X}\stackrel{\bigtriangledown \omega}{\rightarrow}\Omega_{M}|_{X}]$ gives the symmetric obstruction theory 
$E_X$. 

Let $\CC^*$ act on $\aA^n$ by setting the degree of $x_i$ to be $r_i\in \ZZ$. Each $f_i$ is homogeneous with respect to these degrees and denote the degree of $f_i$ by $n_i\in \ZZ$. Then the zero locus $X$ admits a $\CC^*$ action.
If we let $\CC^*$ act on the tangent $T_M$ by letting the degree of $\frac{\partial}{\partial x_i}$ to be $n_i$, then $H(\omega)\to \LL_X$ defines a $\CC^*$ perfect equivariant obstruction theory. 

If $n_i=-r_i$, then the form $\omega$ is an invariant element of $\Gamma(M, \Omega_M)$.  Then we have an ``equivariant symmetric" obstruction theory. 
\end{Example}

The following result is due to Proposition 2.6 of Behrend-Fantechi \cite{BF2}.

\begin{proposition}
Let $X$ be an affine $\CC^*$-scheme with a fixed point $P$ and let $n=\dim T_X|_{P}$.  If $X$ is endowed with a symmetric equivariant obstruction theory $E_X\to \LL_X$.  Then there exists an invariant affine open neighbourhood $U$ of $P$ in $X$, an equivariant closed embedding $U\hookrightarrow M$ into a smooth $\CC^*$-scheme $M$ of dimension $n$ and an invariant almost closed one form $\omega$ on $M$ such that $X=Z(\omega)$.  
\end{proposition}

We actually don't need $X$ to admit an equivariant ``symmetric" obstruction theory, only a $\CC^*$-equivariant obstruction theory $E_X\to \LL_X$.  

\begin{theorem}
Let $X$ be a scheme which admits a symmetric obstruction theory $E_X$. Furthermore assume that there is a  $\CC^*$ action on $X$ with fixed point scheme $F\subset X$,  such that  $E_X\to \LL_X$ is a $\CC^*$-equivariant  obstruction theory.  
Then 
$$\chi(X, \nu_X)=\chi(F, \nu_{X}|_{F}),$$
where $\nu_X|_{F}$ is the restriction of $\nu_X$ to $F$.
Moreover if $X$ is proper, then 
$\int_{[X]^\vir}1=\chi(X, \nu_X)=\chi(F, \nu_{X}|_{F})$. 
\end{theorem}
\begin{proof}
From the property of the Behrend function, $\nu_X$ is constant on the nontrivial $\CC^*$ orbits, and the Euler characteristic of a $\CC^*$-scheme without fixed points is zero. So the only contribution comes from the fixed point locus and 
$\chi(X, \nu_X)=\chi(F, \nu_{X}|_{F})$. If $X$ is proper, the last statement is Behrend's theorem Theorem \ref{Behrend}. 
\end{proof}

\subsection{Kiem-Li construction}
We denote the $\CC^*$ action on $X$ by:
$$\mu: \CC^*\times X\to X.$$
Let $\lambda$ be the parameter of the group $\CC^*$ and consider the following  vector field
$$v: X\to T_X$$
given by $v=\frac{d}{d\lambda}(\mu(\lambda\cdot x))|_{\lambda=1}$. 
The zero locus of such vector field is $F\subset X$, the fixed point locus. The vector field defines a cosection 
$$\sigma: \Omega_{X}\to \cO_{X}$$
by taking dual in the sense of \cite{KL}. The degenerate locus $D(\sigma)$ is exactly the fixed point locus $F\subset X$.

The obstruction theory $E_{X}$ defines a bundle stack 
$\mathbf{E}:=h^1/h^0((E_{X})^{\vee})$ in the sense of Behrend-Fantechi in Section 2 of \cite{BF}, such that 
$h^1(\mathbf{E})=\Omega_{X}$. 
Let $U:=X\setminus F$. Then $U$ is an open subset of $X$. 
The cosection $\sigma$ gives a surjective morphism 
\begin{equation}\label{cosection.bundle.map}
\Omega_{X}|_{U}=h^1((E_{X})^{\vee})|_{U}\to \cO_{U}.
\end{equation}
This surjective morphism (\ref{cosection.bundle.map}) induces a morphism from the bundle stack 
$\mathbf{E}|_{U}$ to $\CC_U$. 
Let 
$$\mathbf{E}(\sigma):=\mathbf{E}|_{F}\cup \ker(\mathbf{E}|_{U}\to\CC_U).$$
Let $\cc_{X}$ be the intrinsic normal cone of $X$.
As proved in \cite{KL}, the intrinsic normal cone $\cc_{X}\subset \mathbf{E}(\sigma)$.
Applying the localized Gysin map in Section 3 of \cite{KL},  we get the zero dimensional localized virtual cycle 
$$s^{!}_{\mathbf{E},\sigma}([\cc_{X}])=[F]^{\vir}_{\loc,\KL}\in A_0(F).$$
Let $\iota: F\to X$ be the inclusion. 
This cycle satisfies $\iota_*([F]^{\vir}_{\loc,\KL})=[X]^{\vir}$ if $X$ is projective.

\begin{definition}
If the  scheme or DM stack $F$ is proper, then  the virtual localized invariant of $X$ is defined by
$$\int_{[F]^{\vir}_{\loc,\KL}}1.$$
\end{definition}

\subsection{Cone decomposition}

Recall that we fix an embedding 
$i: X\hookrightarrow M$
of $X$ to a smooth scheme $M$.   Hence there is a surjective morphism:
$\phi: \Omega_{M}|_{X}\to \Omega_X$.  From Section \ref{symmetric_obstruction_theory}, 
the obstruction cone $C\subset \Omega_{M}|_{X}$ is a conic Lagrangian cycle in $\Omega_{M}$, which is  proved in Theorem 4.9 of \cite{Behrend}.  

\begin{definition}\label{Decomposition_C}
Define
$$C_2:=\overline{C|_{X\setminus F}}.$$
and 
$$[C_1]:=[C]-[C_2]$$
as cycles. 
\end{definition}
\begin{remark}
The cone $C_1$ is supported on  the fixed point locus  $F$ of $X$.  
\end{remark}

\subsection{Kiem-Li localized invariant via bundle stack}

First we prove the following:
\begin{proposition}
The surjective morphism $\phi: \Omega_{M}|_{X}\to \Omega_{X}$ induces a surjective morphism 
$$\Phi: \Omega_{M}|_{X}\to\mathbf{E}$$
to the bundle stack 
$\mathbf{E}$. 
\end{proposition}
\begin{proof}
Let $\mathfrak{e}: \mathfrak{X}\to X$ be an \'etale  open covering. The the pullback 
$$\mathfrak{e}^*E_X\cong [F^{-1}\to F^0]$$
has a locally free resolution, and $F^{-1}, F^0$ are locally free coherent sheaves. 
Then $\mathfrak{e}^*E_X^\vee\cong [F_{0}\to F_1]$, with $F_{0}=(F^0)^\vee, F_1=(F^{-1})^\vee$ and  $\mathfrak{e}^*\mathbf{E}=[F_1/F_0]$. 

Hence we have the following exact sequence
$$F_0\to F_1\to \mathfrak{e}^*\ob\to 0.$$
Suppose that the ranks of $F_0$ and $F_1$ are $r_0$ and $r_1$. 
Let $g$ be the number of generators of the sheaf $\ob$ on one local chart around a point $P$.
 
Locally we can take $r_1=g$ without loss of generality.
(This is because the rank of the first arrow is $r_1-g$ at $P$, so we may choose a rank $r_1-g$ subbundle $S$ on which the map is of full rank (possibly after shrinking our open set). Then dividing out by the acyclic complex $S\to S$ we get the claim.)
 
Now our vector bundle $\Omega_M|_X$ surjects onto $\ob$. Since it is free, we can lift to a map $\Omega_M|_{X}\to F_1$.
Then the composition $\Omega_M|_{X}\to F_1\to \ob$ is onto, so in particular the first map must have rank $\ge g$, so it must be onto, so $\Omega_{M}|_{X}\to F_1$ is onto in the local chart and $\mathfrak{e}^*\Omega_{M}|_{X}$ maps onto the stack 
$\mathfrak{e}^*\mathbf{E}=[F_1/F_0]$. 
The construction is canonical so it induces a surjective morphism 
$\Phi: \Omega_{M}|_{X}\to \mathbf{E}$. 
\end{proof}

We construct the following diagram:
\begin{equation}
\xymatrix{
\Omega_{M}|_{X}\ar[r]^{\Phi}\ar[dr]& \mathbf{E}\ar[d]\ar[dr]^{\sigma} &\\
&\Omega_X\ar[r]^{\sigma}&\cO_X
}
\end{equation}
where $\sigma$ is the cosection map. 
To simplify notation,  we denote by  $V:=\Omega_{M}|_{X}$. 
Let 
$$V(\sigma):=V|_{F}\oplus \ker(V|_{U=X\setminus F}\to \cO_X|_{U}).$$
The Lagrangian cone $C$ lies in $V(\sigma)$.  Kiem-Li's localized virtual cycle 
$[F]^{\vir}_{\KL}$ is given by the localized Gysin map
\begin{equation}\label{KL_Gysin}
s^{!}_{V(\sigma), \sigma}([C])=[F]^{\vir}_{\loc, \KL}
\end{equation}
constructed in Section 3 of \cite{KL}. 
The following result is easily from the property of localized Gysin map. 
\begin{proposition}
$$s^{!}_{V(\sigma), \sigma}([C])=s^{!}_{V(\sigma), \sigma}([C_1])+s^{!}_{V(\sigma), \sigma}([C_2]).$$
\end{proposition}

\subsection{Kiem-Li localized invariant via Lagrangian intersection}

\subsubsection{Type one cone:}
The cone $C_1$ inside $\Omega_{M}|_{X}$ supports on the fixed point loci $F\subset X$.  So $C_1$ lies inside 
$V|_{F}$.  The scheme $F$ is closed and proper. Hence from Kiem-Li \cite{KL}, the localized Gysin map is just the usual Gysin map and 
\begin{equation}\label{KL_C1}
s^{!}_{V(\sigma), \sigma}([C_1])=s^{!}_{V|_{F}}([C_1]).
\end{equation}

\subsubsection{Type two cone:}
Each irreducible component of the cone $C_2$ supports on a $\CC^*$-invariant subscheme of $X$. 
Using idea in the Appendix in \cite{KL}, we prove that the localized Gysin map $s^{!}_{V(\sigma), \sigma}([C_2])$ is still a Lagrangian intersection of two Lagrangian cycles inside 
$\Omega_{M}|_{X}$. 

We use a similar idea as in Section 4 of \cite{JThomas}. 
Choose a Hermitian metric on the vector bundle $\Omega_M$, and choose a smooth function
$$\psi:  M\to \RR$$ 
such that 
$\{x\in M: \psi(x)\leq t\}$ is compact for $t\in \RR$. 
We also choose a smooth  function $\rho: \RR\to \RR_{\geq 0}$ such that 
$$
\rho(r)=
\begin{cases}
0, & r\leq 1;\\
\epsilon r, & r\geq 2
\end{cases}
$$
Let $\Gamma:=\Gamma(d\rho\psi)$ be the graph of the differential of the function $\rho\psi$. 
Then $\Gamma$ is a small perturbation of the zero section of $\Omega_{M}$. 
The graph $\Gamma$  is a real Lagrangian cycle inside $\Omega_{M}$.  
From Theorem 9.5.3 and Theorem 9.7.11 of  \cite{KS}, the Lagrangian intersection $I(C_2, \Gamma)$ makes sense. 

\begin{theorem}\label{KL2_C2_Gamma}
The Lagrangian intersection $C_2\cap \Gamma$ supports on the fixed point loci $F\subset X$ and the intersection number  $I(C_2, \Gamma)$ is the same as Kiem-Li localized invariant $s^{!}_{V(\sigma),\sigma}([C_2])$. 
\end{theorem}
\begin{proof}
We can choose the smooth function $\psi:  M\to \RR$ such that $\psi$ is nonzero on $X\setminus F$.  
The composition 
$$\Omega_{M}|_{X}\to \Omega_{X}\to \cO_{X}$$
is the contraction with the Euler vector field associated with the $\CC^*$ action. 
When restricting to the graph $\Gamma$, this gives a function 
$$\epsilon\psi$$
outside the neighbourhood $\{\psi\leq 2\}$. 
This is nonzero, hence the intersection of $C_2$ with $\Gamma$ can not support on  the kernel of
$\Omega_{M}|_{X}\to \cO_X$ of the cosection on the locus $X\setminus F$.  But the cone $C_2$ lies in the kernel $\Omega_{M}|_{X}\to \cO_X$. Hence  $C_2\cap \Gamma$ must support on the locus $F\subset X$. 

The intersection number $I(C_2, \Gamma)=s^{!}_{V(\sigma),\sigma}([C_2])$ is a global result of Proposition A.1 of the Appendix in \cite{KL}. 
\end{proof}

\subsection{Behrend VS Kiem-Li}

Recall that in Section \ref{weighted_Euler_char} there is  a canonical  integral cycle $\mathfrak{c}_X$ for the scheme or DM stack 
$X$. The cycle $\mathfrak{c}_X$ has a decomposition:
$$\mathfrak{c}_X=\mathfrak{c}_1+\mathfrak{c}_2,$$
where 
$\mathfrak{c}_2=\overline{\mathfrak{c}_X|_{X\setminus F}}$ and $\mathfrak{c}_1=\mathfrak{c}_X\setminus \mathfrak{c}_2$. 
By the uniqueness of the integral cycle $\mathfrak{c}_X$, the integral cycles $\mathfrak{c}_1$ and $\mathfrak{c}_2$ are unique.

Recall that the obstruction cone $C\subset \Omega_{M}|_{X}$ is a conic Lagrangian cycle and we have the cone decomposition 
$C=C_1+C_2$ in Definition \ref{Decomposition_C}.   Let 
$\mathfrak{L}_{X}(\Omega_M)$ be the subgroup of $Z_n(\Omega_M)$ generated by conic Lagrangian prime cycles which support on $X$, where $n=\dim(M)$. 

\begin{proposition}
We have:
$$L(\mathfrak{c}_1)=C_1, \quad L(\mathfrak{c}_2)=C_2.$$
\end{proposition}
\begin{proof}
The subgroup $\mathfrak{L}_{X}(\Omega_M)$ is generated by conic Lagrangian prime cycles, and  from \cite{Ginzburg}, also \cite{Behrend}, every irreducible conic Lagrangian cycle in $\Omega_{M}|_{X}$ is of the form $N^*_{Z/M}$.  Our cones $C_1, C_2$ are conic Lagrangian in $\Omega_{M}|_{X}$.  So each irreducible component of $C_1$ (resp. $C_2$) is of the form $N^*_{Z/M}$, where in the case of $C_2$, $N^*_{Z/M}$ is supported on $Z\subset \overline{X\setminus F}$, and in the case of $C_1$, $N^*_{Z/M}$ is supported on $Z\subset M$.  Thus the results follows from the isomorphisms in (\ref{cycle_Lagrangian}) and (\ref{Lagrangian_cycle}).
\end{proof}

\begin{definition}
Define
$$\nu_1=\eu(\mathfrak{c}_1), \quad  \nu_2=\eu(\mathfrak{c}_2).$$
\end{definition}
\begin{remark}
From the property of Euler obstruction, $\eu(\mathfrak{c}_N)=\eu(\mathfrak{c}_1)+\eu(\mathfrak{c}_2)$. 
Hence the Behrend function $\nu_X=\nu_1+\nu_2$.
\end{remark}

\begin{theorem}\label{Kai_KL1}
We have:
$$\chi(F, \nu_1|_{F})=s^{!}_{V(\sigma),\sigma}([C_1]).$$
\end{theorem}
\begin{proof}
The Lagrangian cycle $[C_1]$ only supports on $F$. From (\ref{KL_C1}), $s^{!}_{V(\sigma),\sigma}([C_1])=s^{!}_{V|_{F}}([C_1])$, the usual Lagrangian intersection of $C_1$ with the zero section of $V|_{F}$.  In the case that $M$ is a scheme, or a finite group quotient stack,  or a finite gerbe over a scheme, this is Behrend's Proposition 4.6 in \cite{Behrend}.  In the DM stack case, it is due to Maulik and Treumann in \cite{MT}, where the main result is the generalization of  Kashiwara index theorem to orbifolds. 
\end{proof}

Recall our Lagrangian cycle $\Gamma\subset \Omega_{M}$, which is a small perturbation of the zero section. 
\begin{theorem}\label{Kai_KL2}
We have:
$$\chi(F, \nu_2|_{F})=I(C_2, \Gamma).$$
\end{theorem}
\begin{proof}
In the case that $M$ is a scheme, this is the global index theorem due to  Theorem 9.7.11 of Kashiwara-Schapira \cite{KS} since the characteristic cycle of $\nu_2$ is $C_2$.  In the DM stack case, it is still due to Maulik and Treumann in \cite{MT}.  
\end{proof}

So we finally have:

\begin{theorem}\label{main_theorem_2}
If the  $\CC^*$ fixed locus $F$ is proper, then 
we have:
$$\chi(F, \nu_X|_{F})=\int_{[F]^\vir_{\loc, \KL}}1.$$
\end{theorem}
\begin{proof}
Combining Theorem \ref{KL2_C2_Gamma},  Theorem \ref{Kai_KL1} and Theorem \ref{Kai_KL2}, the result follows. 
\end{proof}
\begin{remark}
Actually to prove this main result we don't need to split cones as in the proof in Section 4 in \cite{JThomas}.  To make things clearer, we do it here. 
\end{remark}

\subsection{$\CC^*$-action with Isolated fixed points}

In this section we assume that  the $\CC^*$-action on the scheme $X$ has isolated fixed point set.  In Theorem 3.4 of \cite{BF},  Behrend and
Fantechi proved the  following result: 
 
\begin{proposition}
Let $X$ be a scheme with a $\CC^*$ action so that $(E_X\to \LL_X, \Theta)$ is a $\CC^*$-equivariant symmetric obstruction theory. 
Suppose that the $\CC^*$ action only has isolated fixed point set $F$. Let $P\in F$. Then 
$$\nu_X(P)=(-1)^{\dim(T_X|_{P})},$$
where $T_X|_{P}$ is the tangent space of $X$ at $P$. 
\end{proposition}
\begin{proof}
The formula in the theorem is a local issue question. 
We prove the  result  using Theorem \ref{main_theorem_2}. 
The symmetric obstruction theory $E_X=[E^{-1}\to E^0]$ is given by a global resolution. 
We may use $\CC^*$ equivariant embedding
$$P\hookrightarrow X\subset M,$$
where we assume that the dimension of $M$ is the same as $\dim(T_X|_{P})$. 
Using deformation to normal cone $C_{X/M}$ and the deformation invariance of virtual class, we may use the zero section $[M]$ to do the Kiem-Li localization. 
So in this case we have a vector bundle $\Omega_M$ and a cosection map
$\sigma: \Omega_M\to \cO_M$ such that $\sigma^{-1}(0)=P$. Then the result comes from 
Example 2.4 in \cite{KL}. 
\end{proof}

\bibliographystyle{amsplain}

\end{document}